\documentclass[10pt]{amsart}

\usepackage{amsmath,amsthm,amsfonts,amscd,amssymb,soul}
\usepackage{hyperref,wasysym}
\usepackage{verbatim}
\usepackage{amssymb}
\usepackage{graphicx}
\usepackage{color}

\newtheorem{thm}{Theorem}[section]
\newtheorem{cor}[thm]{Corollary}

\newtheorem{prop}[thm]{Proposition}

\theoremstyle{definition}

\theoremstyle{remark}
\newtheorem{rem}{Remark}[section]

\begin{document}

\title{On permutation-invariant construction of glued lattices}

\author{Maria Fernanda Zordan Bonini}
\author{Lenny Fukshansky}

\address{Departamento de Matem\'atica, S\~ao Paulo State University, R. Crist\'ov\~ao Colombo, 2265 - Jardim Nazareth, S\~ao Jos\'e do Rio Preto, 15054-000, Brazil}
\email{maria.bonini@unesp.br}
\address{Department of Mathematics, 850 Columbia Avenue, Claremont McKenna College, Claremont, CA 91711, USA}
\email{lenny@cmc.edu}


\begin{abstract} Given a permutation $\tau$ on $n$ letters, we consider lattices spanned by an orbit of one vector $\boldsymbol x$ in $\mathbb R^n$ under the action of $\tau$ by permutation of the coordinates. Such lattices generalize the important class of cyclic lattices and have previously been studied in~\cite{perm}, where a bound on their rank was established. We prove a sufficient condition on $\boldsymbol x$ for this bound to be achieved. We further investigate the structure of such permutation-invariant lattices, proving that they are glued by the permuted vector from the orthogonal cyclic blocks and giving a determinant formula for the lattice in terms of determinants of these blocks and the norm of the permuted vector. In the case $\boldsymbol x$ is an integer vector, these blocks are sublattices of the root lattices $A_k$ in respective dimensions with root lattices themselves and their glued direct sums also realizable by this construction. We also exhibit a glued construction of permutation-invariant algebraic integral lattices from collections of cyclic number fields. Finally, we prove a strengthened version of a previous result of~\cite{lf_ek} on a related construction of well-rounded lattices spanned by sets of algebraic conjugates.
\end{abstract}

\maketitle

\def\A{{\mathcal A}}
\def\B{{\mathcal B}}
\def\C{{\mathcal C}}
\def\D{{\mathcal D}}
\def\F{{\mathcal F}}
\def\x{{\mathcal H}}
\def\I{{\mathcal I}}
\def\J{{\mathcal J}}
\def\K{{\mathcal K}}
\def\L{{\mathcal L}}
\def\M{{\mathcal M}}
\def\N{{\mathcal N}}
\def\O{{\mathcal O}}
\def\R{{\mathcal R}}
\def\s{{\mathcal S}}
\def\V{{\mathcal V}}
\def\W{{\mathcal W}}
\def\X{{\mathcal X}}
\def\Y{{\mathcal Y}}
\def\H{{\mathcal H}}
\def\Z{{\mathcal Z}}
\def\OO{{\mathcal O}}
\def\BB{{\mathbb B}}
\def\cee{{\mathbb C}}
\def\EE{{\mathbb E}}
\def\Nn{{\mathbb N}}
\def\pee{{\mathbb P}}
\def\que{{\mathbb Q}}
\def\real{{\mathbb R}}
\def\zed{{\mathbb Z}}
\def\hyp{{\mathbb H}}
\def\aa{{\mathfrak a}}
\def\HH{{\mathfrak H}}
\def\qbar{{\overline{\mathbb Q}}}
\def\eps{{\varepsilon}}
\def\ahat{{\hat \alpha}}
\def\bhat{{\hat \beta}}
\def\gt{{\tilde \gamma}}
\def\h{{\tfrac12}}
\def\b1{{\boldsymbol 1}}
\def\be{{\boldsymbol e}}
\def\bei{{\boldsymbol e_i}}
\def\bff{{\boldsymbol f}}
\def\ba{{\boldsymbol a}}
\def\bb{{\boldsymbol b}}
\def\bc{{\boldsymbol c}}
\def\bm{{\boldsymbol m}}
\def\bn{{\boldsymbol n}}
\def\bk{{\boldsymbol k}}
\def\bi{{\boldsymbol i}}
\def\bl{{\boldsymbol l}}
\def\bq{{\boldsymbol q}}
\def\bu{{\boldsymbol u}}
\def\bt{{\boldsymbol t}}
\def\bs{{\boldsymbol s}}
\def\bv{{\boldsymbol v}}
\def\bw{{\boldsymbol w}}
\def\bx{{\boldsymbol x}}
\def\bX{{\boldsymbol X}}
\def\bz{{\boldsymbol z}}
\def\bwy{{\boldsymbol y}}
\def\bY{{\boldsymbol Y}}
\def\bL{{\boldsymbol L}}
\def\baa{{\boldsymbol\alpha}}
\def\bbb{{\boldsymbol\beta}}
\def\bgg{{\boldsymbol\gamma}}
\def\bet{{\boldsymbol\eta}}
\def\bxi{{\boldsymbol\xi}}
\def\bo{{\boldsymbol 0}}
\def\bol{{\boldkey 1}_L}
\def\ep{\varepsilon}
\def\p{\boldsymbol\varphi}
\def\q{\boldsymbol\psi}
\def\rank{\operatorname{rank}}
\def\aut{\operatorname{Aut}}
\def\lcm{\operatorname{lcm}}
\def\sgn{\operatorname{sgn}}
\def\spn{\operatorname{span}}
\def\md{\operatorname{mod}}
\def\Norm{\operatorname{Norm}}
\def\dim{\operatorname{dim}}
\def\det{\operatorname{det}}
\def\Vol{\operatorname{Vol}}
\def\rk{\operatorname{rk}}
\def\Gal{\operatorname{Gal}}
\def\WR{\operatorname{WR}}
\def\WO{\operatorname{WO}}
\def\GL{\operatorname{GL}}
\def\pr{\operatorname{pr}}
\def\Tr{\operatorname{Tr}}
\def\dd{\partial}
\def\itt{\operatorname{int}}
\def\Ar{\operatorname{Area}}
\def\Aut{\operatorname{Aut}}
\def\ot{\operatorname{o}_{\tau}}

\section{Introduction}
\label{intro}

Let $n \geq 2$ and write $S_n$ for the permutation group on $n$ letters. This group acts on $\real^n$ by permutation of the coordinates. Taking the standard representation of $S_n$ in $\GL_n(\zed)$ via permutation matrices, we can also view this action as left multiplication. Fix an element $\tau \in S_n$, then for each $\bx \in \real^n$ we write $\tau(\bx)$ for the image of $\bx$ under the action by $\tau$. A lattice $L \subset \real^n$ (i.e., a discrete subgroup) is called {\it $\tau$-invariant} if $\tau(L) = L$, which is equivalent to saying that $\tau$ is in $\Aut(L)$, the automorphism group of $L$, defined as
$$\Aut(L) = \left\{ U \in \GL_n(\zed) : (U\bx)^{\top} (U\bwy) = \bx^{\top} \bwy,\ \forall\ \bx,\bwy \in L \right\}.$$
We refer to all such lattices as {\it permutation-invariant}; they generalize the class of {\it cyclic} lattices, which are the $\tau$-invariant lattices for $\tau = (1\ \dots\ n)$. Cyclic lattices were introduced by D. Micciancio in the context of lattice cryptography~\cite{mic1}, \cite{mic2} and then studied by many other authors. Permutation-invariant lattices were studied in~\cite{perm} with a special focus on lattices of the form
$$L_{\tau}(\bx) = \spn_{\zed} \left\{ \tau^k(\bx) : 0 \leq k \leq |\tau|-1 \right\},$$
where $\bx \in \real^n$ and $|\tau|$ is the order of the element $\tau$ in the group $S_n$. We will call such lattices {\it simple $\tau$-invariant}, generated by permutations of $\bx$ under $\tau$; this is a generalization of the {\it simple cyclic} lattices that were studied in~\cite{lf_dk}. Simple cyclic and, more generally, simple permutation-invariant lattices are especially useful for cryptographic applications due to their very compact description as orbits of just one vector under a given permutation. A just-announced breakthrough result by D. Wan~\cite{d_wan} asserts that the Shortest Vector Problem (SVP) is NP-hard on cyclic (and hence, on permutation-invariant) lattices, establishing them as excellent candidates for secure crypto-scheme design. Further, cyclic lattices are naturally analogous to cyclic codes construction in coding theory (see, e.g. Chapters~7-8 of~\cite{macwillians_sloane}); similarly, permutation-invariant lattices are analogous to the more recently introduced class of generalized quasi-cyclic codes (see, e.g, \cite{gqc}). The goal of this paper is to study the structure of simple permutation-invariant lattices.

For each $\bx \in \real^n$, the lattice $L_{\tau}(\bx)$ is $\tau$-invariant by construction, however its rank can vary. Define
$$\ot = \max \left\{ \rk L_{\tau}(\bx) : \bx \in \real^n \right\},$$
then Lemma~2.1 of~\cite{perm} asserts that $\ot \leq n-\ell+1$, where $\ell$ is the number of cycles in a disjoint cycle decomposition of $\tau$; the bound is achieved if these cycles have pairwise relatively prime lengths. The first question we want to address is for which vectors $\bx \in \real^n$ is $\rk L_{\tau}(\bx) = \ot$?

Let $\tau = c_1 \cdots c_{\ell}$ be the disjoint cycle decomposition for $\tau$ with $k_j = |c_j|$ for each $1 \leq j \leq \ell$, ordered so that $k_1 \leq \dots \leq k_\ell$. Then $\sum_{j=1}^{\ell} k_j = n$, and for a vector $\bx \in \real^n$ we write
$$\bx = \begin{pmatrix} \bx_1 \\ \vdots \\ \bx_{\ell} \end{pmatrix},\text{ so } \tau(\bx) = \begin{pmatrix} c_1(\bx_1) \\ \vdots \\ c_{\ell}(\bx_{\ell}) \end{pmatrix},$$
where each block $\bx_j = (x_{j1}, \dots, x_{j k_j})^{\top} \in \real^{k_j}$. For each $1 \leq j \leq \ell$, define
$$d(\bx_j) = \dim_{\que} \spn_{\que} \{ x_{j1},\dots,x_{jk_j} \},$$
and analogously for any vector with real coordinates. With this notation, we can state our first observation.

\begin{prop} \label{ind} If $d(\bx_j) \geq k_j$ for each $1 \leq j \leq \ell$, then $\rk L_{\tau}(\bx) = \ot$.
\end{prop}

\noindent
We prove Proposition~\ref{ind} in Section~\ref{independent}. The hypothesis of this proposition is a sufficient condition for $\rk L_{\tau}(\bx)$ to be equal to $\ot$, however it is not necessary. In fact, some of the constructions of such lattices we show in Section~\ref{root} do not satisfy this condition. Our main result describes the structure of simple permutation-invariant lattices.

\begin{thm} \label{glue} Let $2 \leq k_1 < \dots < k_{\ell}$ be pairwise relatively prime integers, $n = \sum_{i=1}^{\ell} k_i$, and $\tau = c_1 \dots c_{\ell} \in S_n$, where $c_1,\dots,c_{\ell}$ are disjoint cycles of lengths $k_1,\dots,k_{\ell}$, respectively. Let $\bx \in \real^n$ be a vector such that $\rk L_{\tau}(\bx) = n-\ell+1$. Then
\begin{equation}
\label{Ltaux}
L_{\tau}(\bx) = \spn_{\zed} \left\{ \bx, \bigoplus_{i=1}^{\ell} M_i \right\},
\end{equation}
where each of the mutually orthogonal blocks $M_i$ is a cyclic lattice in the corresponding space $\real^{k_i}$ embedded as the respective coordinate subspace of $\real^n$ and $\bx \notin \bigoplus_{i=1}^{\ell} M_i$. Further, $\rk M_i = k_i-1$ for each $1 \leq i \leq \ell$ and 
$$M_i \subset \b1_{k_i}^{\perp} := \left\{ \bz \in \real^{k_i} : \sum_{j=1}^{k_i} z_j = 0 \right\}.$$
Conversely, given a collection of simple cyclic lattices $M_i \subset \b1_i^{\perp}$, $1 \leq i \leq \ell$ with $\rk M_i = k_i-1$, there exist infinitely many vectors $\bx \in \real^n$ for which~\eqref{Ltaux} holds so that $\bx \notin \bigoplus_{i=1}^{\ell} M_i$ and $\rk L_{\tau}(\bx) = n-\ell+1$. 
\end{thm}

\noindent
We refer to $L_{\tau}(\bx)$ as the lattice {\it glued} from the blocks $M_i$ with the {\it glue vector}~$\bx$ (see Section~4.3 of~\cite{conway_sloane} for the details of gluing theory). Hence, the $\tau$-permutation orbit of the glue vector $\bx$ results in the glued lattice construction. We prove Theorem~\ref{glue} in Section~\ref{root}, where we also describe in details the construction of the blocks $M_i$ from the glue vector $\bx$. Further, we give a formula for the determinant of $L_{\tau}(\bx)$ in terms of the determinants of the blocks and norm of $\bx$ as well as briefly discuss the successive minima of $L_{\tau}(\bx)$ in Remark~\ref{invariants}. In the case when $\bx \in \zed^n$, these blocks lie in the {\it root lattices}
$$A_{k_i-1} = \left\{ \bz \in \zed^{k_i} : \sum_{j=1}^{k_i} z_j = 0 \right\},$$
which are generated by the {\it roots}, i.e., vectors of squared Euclidean norm $2$. We demonstrate a specific example of a glue vector resulting in the blocks being equal to $A_{k_i-1}$, in which case the determinant and successive minima are particularly easy to compute.

On the other hand, we can produce infinite families of simple permutation-invariant lattices satisfying the condition of Proposition~\ref{ind}. A lattice is called {\it integral} if the inner product of any two vectors in it is in $\zed$; a lattice is called {\it algebraic} if it is the image of a free $\zed$-module in a number field under the Minkowski embedding of this number field into a Euclidean space with the trace-induced Euclidean norm. In Section~\ref{nf_sec}, we review the necessary number field notation and describe infinite families of integral simple permutation-invariant algebraic lattices satisfying the hypothesis of Proposition~\ref{ind}. They are also glued together from cyclic algebraic orthogonal blocks, as described in Theorem~\ref{glue}. Finally, in the Appendix (Section~\ref{conj}) we revisit the related construction of lattices spanned by full sets of algebraic conjugates, considered in~\cite{lf_ek} and prove a strengthening of the main result of~\cite{lf_ek}. We are now ready to proceed.
\bigskip

\section{Permuting vectors with linearly independent coordinates}
\label{independent}

In this section, we prove Proposition~\ref{ind}. Let $m=\ot$ and suppose that $\rk L_{\tau}(\bx)$ is smaller than $m$, then there exist integer coefficients $a_0,\dots,a_{m-1}$ such that
$$\sum_{k=0}^{m-1} a_k \tau^k(\bx) = \bo,$$
meaning that
\begin{equation}
\label{sum_cycle-1}
\sum_{k=0}^{m-1} a_k c_j^k(\bx_j) = \bo,\ \forall\ 1 \leq j \leq \ell.
\end{equation}
Grouping together terms with the same coordinate $x_{ji}$ for each $1 \leq j \leq \ell$, $1 \leq i \leq k_j$ in each equation of this linear system, we obtain a system of equations in $x_{ji}$ with coefficients being sums of $a_k$'s. Namely,
$$q_0x_{j1} + q_1x_{j2} + \dots + q_{k_j-1}x_{jk_j} = 0,$$
where $q_r = \sum_{t \equiv r (\md k_j)} a_t$ for each $0 \leq r \leq k_j-1$. For example, if $\ell=2$, $k_1=2$, $k_2 = 3$, our system becomes:
\begin{align*}
(a_0+a_2)x_{11} + (a_1+a_3) x_{12} = 0, (a_1+a_3)x_{11} + (a_0+a_2) x_{12} = 0,\\
(a_0+a_3)x_{21} + a_1 x_{22} + a_2x_{23} = 0, a_1x_{21} + a_2x_{22} + (a_0+a_3)x_{23} = 0,\\
a_2x_{21} + (a_0+a_3)x_{22} + a_1x_{23} = 0.
\end{align*}
Since $d(\bx_j) \geq k_j$ for each $1 \leq j \leq \ell$, we must have all the coefficients (which are sums of $a_k$'s independent of $x_{ji}$'s) in this system equal to $0$. For instance, in the example above, it amounts to
$$a_0+a_2=0,\ a_1+a_3 = 0,\ a_0+a_3 = 0,\ a_1 = 0,\ a_2 = 0.$$
This means that there exist integer coefficients $a_0,\dots,a_{m-1}$, not all zero, such that the corresponding linear system is identically zero for any choice of $\bx \in \real^n$. This means that 
$$\max \left\{ \rk L_{\tau}(\bx) : \bx \in \real^n \right\} < m,$$
contradicting the definition of $\ot$. Hence, we must have $\rk L_{\tau}(\bx) = \ot$ for each 
$$\bx = (\bx_1, \dots, \bx_{\ell})^{\top} \in \real^{k_1} \times \dots \times \real^{k_{\ell}}$$
with $d(\bx_j) \geq k_j$.

\bigskip

\section{Gluing construction}
\label{root}

In this section we describe the structure of simple permutation-invariant lattices in $\real^n$, proving Theorem~\ref{glue}. Let $\tau = c_1 \cdots c_{\ell} \in S_n$ with $k_i = |c_i|$ for each $1 \leq i \leq \ell$ and $\gcd(k_i,k_j) = 1$ for $i \neq j$. Let
$$\bx = \left( \bx_1,\dots, \bx_{\ell} \right)^{\top} \in \real^n \text{ with each } \bx_i \in \real^{k_i}$$
be such that $\rk L_{\tau}(\bx) = \ot$. Since the cycle lengths are relatively prime, the Chinese Remainder Theorem implies that for each $1 \leq i \leq \ell$ and $1 \leq t(i) \leq k_i-1$ there exists a power $0 \leq m_{it(i)} < k_1 \cdots k_{\ell} = |\tau|$ such that $c_j^{m_{it(i)}} = c_j^{k_j}$ for all $j \neq i$ and $c_i^{m_{it(i)}} = c_i^{t(i)}$. This means that the vectors $\bx$ and $\tau^{m_{it(i)}}(\bx)$ will have different coordinates only in the $i$-th block. Then define
\begin{equation}
\label{diff_vector}
\bz_{it(i)} = \bx - \tau^{m_{it(i)}}(\bx) = \left( \bo,\dots,\bo, \bx_i-c_i^{t(i)}(\bx_i), \bo, \dots,\bo \right) \in L_{\tau}(\bx).
\end{equation}
Notice that the coordinates of each $\bx_i-c_i^{t(i)}(\bx_i)$ (and hence, the coordinates of $\bz_{it(i)}$) add up to $0$. In particular, this means that for $\bx \in \zed^n$, each $\bz_{it(i)} \in A_{k_i-1}$, the corresponding root lattice. All the vectors $\bz_{it(i)}$ are linearly independent and there are a total of $\sum_{i=1}^{\ell} (k_i-1) = n-\ell$ of them. Further,
\begin{eqnarray*}
L_{\tau}(\bx) & = & \spn_{\zed} \left\{ \bx, \bz_{it(i)} : 1 \leq i \leq \ell,\ 1 \leq t(i) \leq k_i-1 \right\} \\
& = & \spn_{\zed} \left\{ \bx, M_i : 1 \leq i \leq \ell \right\},
\end{eqnarray*}
where $M_i = \spn_{\zed} \left\{ \bz_{it(i)} : 1 \leq t(i) \leq k_i-1 \right\}$ for each $1 \leq i \leq \ell$ is invariant under the cyclic permutation $c_i$, and the blocks $M_i, M_j$ are mutually orthogonal for all $i \neq j$. Thus
\begin{equation}
\label{orth_span}
L_{\tau}(\bx) = \spn_{\zed} \left\{ \bx, \bigoplus_{i=1}^{\ell} M_i \right\}.
\end{equation}
In particular, if $\bx \in \zed^n$, then
$$L_{\tau}(\bx) \subseteq \spn_{\zed} \left\{ \bx, \bigoplus_{i=1}^{\ell} A_{k_i-1} \right\}.$$

In the reverse direction, suppose that for each $1 \leq i \leq \ell$, $M_i \subset \b1_i^{\perp} \subset \real^{k_i}$ is a cyclic lattice of rank $k_i-1$ and let $\bz_{i1},\dots,\bz_{i(k_i-1)}$ be a basis for $M_i$. Let us construct a vector $\bx \in \real^n$ so that~\eqref{orth_span} holds. Take any real numbers $\alpha_1,\dots,\alpha_{\ell}$, and for each $1 \leq i \leq \ell$, define the $i$-th block of $\bx$ to be
$$\bx_i = \left( \alpha_i, \alpha_i + z_{i2},\dots,\alpha_i + \sum_{l=1}^j z_{il},\dots,\alpha_i - z_{i1} \right)^{\top},$$
and let $\bx = (\bx_1,\dots,\bx_{\ell})^{\top}$. Using Chinese Remainder Theorem same way as above, we obtain vectors $(\bo,\dots,\bz_i,\dots,\bo)$ for each $1 \leq i \leq \ell$ as difference vectors of $\bx$ and its permutations by appropriate powers of $\tau$. This gives~\eqref{orth_span}.

\begin{rem} \label{invariants} One advantage of this family of simple permutation invariant lattices $L_{\tau}(\bx)$ is that some of their geometric invariants can be easy to compute. Recall that the {\it determinant} of $L_{\tau}(\bx)$ is the volume of its fundamental domain, which is equal to the absolute value of the determinant of any basis matrix, and the {\it successive minima} $0 < \lambda_1 \leq \dots \leq \lambda_{n-\ell+1}$ of $L_{\tau}(\bx)$ are defined as
$$\lambda_i = \min \left\{ t \in \real_{>0} : \dim_{\real} \spn_{\real} \left( \BB_{n-\ell+1}(t) \cap L_{\tau}(\bx) \right) \geq i \right\},$$
where $\BB_{n-\ell+1}(t)$ is a ball of radius $t$ centered at the origin in $\real^{n-\ell+1}$. Then determinant of $L_{\tau}(\bx)$ can be computed as the product of determinants of the cyclic orthogonal blocks times the norm of the projection of $\bx$ in the direction orthogonal to the hyperplane $V$ spanned by all those blocks, i.e.,
\begin{equation}
\label{det_formula}
\det L_{\tau}(\bx) = \left( \prod_{i=1}^{\ell} \det M_i \right) \|\bx\| \sin \theta,
\end{equation}
where $\theta$ is the angle $\bx$ makes with $V$. Further, if $\theta \in [\pi/3,\pi/2]$ (a sufficient, not necessary condition), then the successive minima of $L_{\tau}(\bx)$ are the successive minima of the orthogonal blocks and $\|\bx\|$.
\end{rem}

Let us consider a particular example. Let
$$\bx = \left( \be^1_1,\dots,\be^{\ell}_1 \right)^{\top} \in \zed^n,$$
where $\be^i_1$ is the first standard basis vector in $\real^{k_i}$. Then the difference vector $\bz_{it(i)}$ as in~\eqref{diff_vector} has first coordinate in the $i$-th block equal to $1$, $t(i)$-th coordinate in the $i$-th block equal to $-1$ and the rest of the coordinates equal to $0$. For each $1 \leq i \leq \ell$, the collection of $i$-th blocks of vectors $\bz_{it(i)}$ spans the root lattice $A_{k_i-1}$ with $\bz_{it(i)}$ being the spanning roots, and so
\begin{equation}
\label{A_blocks}
L_{\tau}(\bx) = \spn_{\zed} \left\{ \bx, \bigoplus_{i=1}^{\ell} A_{k_i-1} \right\}
\end{equation}
is glued from the root lattice blocks by the glue vector $\bx$. The situation is special when $\ell=2$. In this case, the glue vector~$\bx$ is also a root, as it has two coordinates equal to $1$ and the rest equal to $0$. Thus, the resulting lattice generated by $n-1$ roots is $L_{\tau}(\bx) = A_{n-1}$. Notice, in particular, that this last statement is independent of the lengths of the cycles $c_1$ and $c_2$. Hence, in the case $\ell=2$, the geometric properties of $L_{\tau}(\bx)$ are well-understood. In the case $\ell > 2$, we can apply~\eqref{det_formula} to obtain
$$\det L_{\tau}(\bx) = \left\{ \left( \prod_{i=1}^{\ell} k_i \right) \times \left( \sum_{i=1}^{\ell} \frac{1}{k_i} \right) \right\}^{1/2} = \left\{ \sum_{i=1}^{\ell} \left( \prod_{j=1, j \neq i}^{\ell} k_j \right) \right\}^{1/2},$$
since determinant of each block $A_{k_i-1}$ is $\sqrt{k_i}$ and the norm of the projection of $\bx$ in the direction orthogonal to the hyperplane $V$ spanned by all those blocks is given by $\sqrt{ \sum_{i=1}^{\ell} \frac{1}{k_i} }$. The first $n-\ell$ successive minima of $L_{\tau}(\bx)$ in this case are equal to $\sqrt{2}$, successive minima of the orthogonal blocks $A_{k_i-1}$, while the last one is $\sqrt{\ell} = \|\bx\|$.
\bigskip

\section{Gluing algebraic integral lattices}
\label{nf_sec}

In this section, we use the gluing construction of Section~\ref{root} in the specific context of lattices coming from number fields, exhibiting a construction of algebraic integral permutation-invariant glued lattices. We start with some notation. Let $f(x) \in \zed[x]$ be an irreducible monic polynomial of degree $n$ so that its Galois group $G(f) = \left< c \right> \leq S_n$ is cyclic and $c = (1\ \dots\ n)$ is the standard generating $n$-cycle. Let $\alpha_1,\dots,\alpha_n \in \qbar$ be the roots of $f(x)$ ordered so that $\alpha_j = c^{j-1}(\alpha_1)$ for $1 \leq j \leq n$. Then $K(f) = \que[x]/(f(x))$ is an extension of $\que$ of degree $n$, thus it is the splitting field of $f(x)$. Let $\sigma_1,\dots,\sigma_n : K(f) \to \cee$ be the embeddings of $K(f)$ with $r_1$ of them real and $2r_2$ complex, coming in conjugate pairs, i.e., $n=r_1+2r_2$. We identify the roots of $f(x)$ with the images of $\alpha_1$ under these embeddings, i.e., $\alpha_j = c^{j-1}(\alpha_1) = \sigma_j(\alpha_1)$ as a complex number, for all $1 \leq j \leq n$. The Minkowski embedding
$$\Sigma_{K(f)} = (\sigma_1,\dots,\sigma_n) : K(f) \hookrightarrow \cee^n$$
takes $K(f)$ into the space $K(f)_{\real} := K(f) \otimes_{\que} \real$, which can be viewed as a subspace of
$$\left\{ (\bx,\bwy) \in \real^{r_1} \times \cee^{2r_2} : y_{r_2+j} = \bar{y}_j\ \forall\ 1 \leq j \leq r_2 \right\} \cong \real^{r_1} \times \cee^{r_2} \subset \cee^n,$$
where in the last containment each copy of $\real$ is identified with the real part of the corresponding copy of $\cee$. This is a Euclidean space where the Euclidean inner product is induced by the number field trace in the sense that
$$\left< \Sigma_{K(f)}(\beta), \Sigma_{K(f)}(\gamma) \right> = \Tr_{K(f)}(\beta \bar{\gamma}) = \sum_{j=1}^n c^{j-1} (\beta \bar{\gamma}) \in \que,$$
for all $\beta,\gamma \in K(f)$. In particular, $\left< \Sigma_{K(f)}(\beta), \Sigma_{K(f)}(\gamma) \right> \in \zed$ for all $\beta,\gamma \in \O_{K(f)}$.

We write
$$\baa = \left( \alpha_1, \dots, \alpha_n \right)^{\top} \in K(f)_{\real}$$
for the image of $\alpha_1$ under $\Sigma_{K(f)}$ and let the $n$-cycle $c$ act on $\baa$ by permutation of the coordinates. Notice that, due to the ordering of the conjugates of $\alpha_1$ we chose,
$$\left< \baa, c^{j-1}(\baa) \right> = \Tr_{K(f)}(\alpha_1 \bar{\alpha_j}) \in \zed,$$
for all $1 \leq j \leq n$. This implies that the inner product of any two integer linear combinations of vectors $c^{j-1}(\baa)$, $1 \leq j \leq n$ is integral. Thus, we have an integral algebraic lattice $\Sigma_K(\M_f)$, where $\M_f = \spn_{\zed} \{ \alpha_1,\dots,\alpha_n \} \subset K$ is a free $\zed$-module spanned by the roots of $f(x)$ in $K$. This observation suggests the following construction.

Let $\tau \in S_n$ be a permutation, given by the disjoint cycle decomposition
$$\tau = c_1 \dots c_{\ell},$$
where each $c_i$ is a $k_i$-cycle, so that $\sum_{i=1}^{\ell} k_i = n$, and $\gcd(k_i,k_j) = 1$ for all $i \neq j$. For each $1 \leq i \leq \ell$, let $f_i(x) \in \zed[x]$ be an irreducible monic polynomial of degree $k_i$ such that $G(f_i) = \left< c_i \right>$ and its roots $\alpha_{i1},\dots,\alpha_{ik_i}$ are $\que$-linearly independent. Notice that for every $k_i \geq 2$ there exist infinitely such polynomials. Write $\baa_i = \left( \alpha_{i1}, \dots, \alpha_{ik_i} \right)^{\top} \in K(f_i)_{\real}$, and let 
$$\bx = \left( \baa_1, \dots, \baa_{\ell} \right) \in K(f_1)_{\real} \times \dots \times K(f_{\ell})_{\real}.$$
Since $\bx$ satisfies the hypothesis of Proposition~\ref{ind}, the lattice $L_{\tau}(\bx)$ has rank $\ot = n-\ell+1$ in the $n$-dimensional space Euclidean space $K(f_1)_{\real} \times \dots \times K(f_{\ell})_{\real}$, it is integral and consists of orthogonal cyclic blocks glued by the vector $\bx$, as described in Theorem~\ref{glue}.

Let us consider an example of this algebraic lattice construction. Let $\tau = (1\ 2) (3\ 4\ 5)$ and take
$$f_1(x) = x^2+x-1,\ f_2(x) = x^3-3x-1.$$
Both of these polynomials have cyclic Galois groups of orders 2 and 3, respectively, and the roots of these polynomials are
$$\alpha_{11} = \frac{-1+\sqrt{5}}{2},\ \alpha_{12} = \frac{-1-\sqrt{5}}{2},$$
and
$$\alpha_{21} = 2 \cos \left( \frac{\pi}{9} \right),\ \alpha_{22} = -2 \cos \left( \frac{2\pi}{9} \right),\ \alpha_{23} = -2 \cos \left( \frac{4\pi}{9} \right),$$
respectively. Taking the glue vector
$$\bx = \left( \frac{-1+\sqrt{5}}{2}, \frac{-1-\sqrt{5}}{2}, 2 \cos \left( \frac{\pi}{9} \right), -2 \cos \left( \frac{2\pi}{9} \right), -2 \cos \left( \frac{4\pi}{9} \right) \right)$$
in the space $\left( \que(\sqrt{5}) \otimes_{\que} \real \right) \times \left( \que\left( \cos \left( \frac{\pi}{9} \right) \right) \otimes_{\que} \real \right) \cong \real^5$, we obtain the lattice $L_{\tau}(\bx) \subset \real^5$ of rank 4 with basis matrix $B = \left( \bx\ \tau(\bx)\ \tau^2(\bx)\ \tau^3(\bx) \right)$ and the corresponding integral Gram matrix
$$C = B^{\top} B = \begin{pmatrix} 9 & -5 & 0 & 4 \\ -5 & 9 & -5 & 0 \\ 0 & -5 & 9 & -5 \\ 4 & 0 & -5 & 9 \end{pmatrix}.$$
Define the mutually orthogonal blocks $M_1 = \spn_{\zed} \left\{ \left( \sqrt{5}, -\sqrt{5}, 0, 0, 0 \right)^{\top} \right\} \subset \b1^{\top}_2 \times \real^3$ and $M_2 = $
$$= \spn_{\zed} \left\{ \begin{pmatrix} 0 \\ 0 \\ 2 \cos \left( \frac{\pi}{9} \right) + 2 \cos \left( \frac{4\pi}{9} \right) \\ -2 \cos \left( \frac{2\pi}{9} \right) - 2 \cos \left( \frac{\pi}{9} \right) \\ -2 \cos \left( \frac{4\pi}{9} \right) + 2 \cos \left( \frac{2\pi}{9} \right),  \end{pmatrix}, \begin{pmatrix} 0 \\ 0 \\ 2 \cos \left( \frac{\pi}{9} \right) + 2 \cos \left( \frac{2\pi}{9} \right) \\ -2 \cos \left( \frac{2\pi}{9} \right) + 2 \cos \left( \frac{4\pi}{9} \right) \\ -2 \cos \left( \frac{4\pi}{9} \right) - 2 \cos \left( \frac{\pi}{9} \right),  \end{pmatrix}\right\} \subset \real^2 \times \b1^{\top}_3,$$
then $L_{\tau}(\bx) = \spn_{\zed} \left\{ \bx, M_1 \oplus M_2 \right\}$ with $\det L_{\tau}(\bx) = 9 \sqrt{15}$ and successive minima $\lambda_1 = \lambda_2 = \sqrt{7}$, $\lambda_3 = \lambda_4 = 2\sqrt{2}$.

\bigskip

\section{Appendix: lattices spanned by algebraic conjugates}
\label{conj}

In this appendix, we want to revisit a construction of algebraic lattices described in~\cite{lf_ek}. Let
$$f(x) = x^n + a_{n-1}x^{n-1} + \dots + a_1x + a_0 \in \zed[x]$$
be a monic irreducible polynomial of degree $n \geq 2$ with roots $\alpha_1, \alpha_2,\dots, \alpha_n$. Let $K$ be its splitting field and $G$ its Galois group over $\que$. Then 
$$d := [K:\que] = |G| \leq n!.$$
Write $G = \left\{ \sigma_1,\dots,\sigma_d \right\}$, then for each $1 \leq i \leq n$ and $1 \leq j \leq d$, $\sigma_j(\alpha_i) = \alpha_{i_j}$ for some corresponding $1 \leq i_j \leq n$. Then
\begin{equation}
\label{sum_a}
-a_{n-1} = \alpha_1 + \dots + \alpha_n \in \zed,
\end{equation}
and
$$\M_f = \spn_{\zed} \left\{ \alpha_1,\dots,\alpha_n \right\} \subset K$$
is a $\zed$-module of rank $\leq n$ generated by all the conjugates of $\alpha$. The module $\M_f$ is closed under the action of $G$, since automorphisms of the Galois extension $K/\que$ simply permute the roots of an irreducible polynomial. Identifying elements of $G$ with the embeddings as in Section~\ref{nf_sec} above, we have the Minkowski embedding
$$\Sigma_K = (\sigma_1,\dots,\sigma_d) : K \hookrightarrow K_{\real}.$$
Writing $\baa$ for the image $\Sigma_K(\alpha) \in K_{\real}$ of $\alpha \in K$, we define the angle $\theta$ between $\baa$ and $\bbb$ is given by
$$\theta = \cos^{-1} \left( \frac{\left< \baa,\bbb \right>}{\|\baa\| \|\bbb\|} \right),$$
for each pair $\alpha,\beta \in K$, where $\|\ \baa \| = \left< \baa, \baa \right>$ is the Euclidean norm.

Define $L_f := \Sigma_K(\M_f)$, which is a Euclidean lattice in $K_{\real}$. Then
$$L_f = \spn_{\zed} \left\{ \baa_1,\dots,\baa_n \right\}.$$
Let $m = \rk L_f$. The {\it minimal norm} of this lattice is $|L_f| := \min \{ \|\bbb\| : \bo \neq \bbb \in L_f \}$, which is the same as the first successive minimum, and the {\it minimal vectors} are vectors with this minimal norm. We say that $L_f$ is {\it well-rounded} (WR) if it has $m$ $\real$-linearly independent minimal vectors. Further, $L_f$ is {\it generic WR} (GWR) if it is WR and the number of minimal vectors is precisely $2m$. Further, let $\B = \{ \bbb_1,\dots,\bbb_m \}$ be an ordered basis for $L_f$, and define a sequence of angles $\theta_1,\dots,\theta_{n-1}$ as follows: each $\theta_i$ is the angle between $\bbb_{i+1}$ and the subspace 
$$\spn_{\real} \{ \bbb_1,\dots,\bbb_i \}.$$
Then each $\theta_i \in [0,\pi/2]$ and we say that $\B$ is a {\it weakly nearly orthogonal} basis if $\theta_i \geq \pi/3$ for each $1 \leq i \leq n-1$. A basis $\B$ is called {\it nearly orthogonal} if every ordering of it is weakly nearly orthogonal. If $L_f$ has such a basis, we say that $L_f$ is a nearly orthogonal lattice. The following result is Theorem~1.5 of~\cite{lf_ek}.

\begin{thm} \label{thm1.5} Let $n \geq 3$ be prime. There exist infinitely many polynomials $f(x) \in \zed[x]$ of degree $n$ so that $L_f$ is a GWR nearly orthogonal lattice of rank $n$ in~$K_{\real}$, where~$K$ is the splitting field of~$f(x)$; further, $L_f$ contains a basis consisting of minimal vectors. For example, we can take $f(x) = x^n + a_{n-1} x^{n-1} + a_0$, where 
\begin{equation}
\label{tri}
|a_0| > \frac{8(n-1)n!}{\sqrt{(n-2)^2 + 16(n-1)} - (n-2)}
\end{equation}
and $|a_{n-1}| > |a_0| + 1$ are integers.
\end{thm}

Here, we prove the following strengthening of this theorem.

\begin{cor} \label{cor1.5} Let $n \geq 3$, not necessarily prime, and let the rest of the notation be as in Theorem~\ref{thm1.5}. There exist infinitely many trinomials $f(x) = x^n + a_{n-1} x^{n-1} + a_0$ satisfying conditions~\eqref{tri} and $|a_{n-1}| > |a_0| + 1$ as in Theorem~\ref{thm1.5} such that $L_f$ is GWR nearly orthogonal with a basis of minimal vectors.
\end{cor}

\proof
The proof is the same as the proof of Theorem~1.5 of~\cite{lf_ek} with one important modification. The only reason the assumption that $n$ is prime was needed was to ensure that rank of $L_f$ is equal to $n$, i.e., that the conjugates $\alpha_1,\dots,\alpha_n$ are $\que$-linearly independent. For this,~\cite{lf_ek} used a theorem of Dubickas (Theorem~1 of~\cite{dub}) that guarantees this linear independence under the conditions $a_{n-1} \neq 0$ and $n$ being prime. The condition $a_{n-1} \neq 0$ is clearly necessary, however $n$ being prime can be replaced by a weaker condition. Specifically, if $a_{n-1} \neq 0$ and $G=S_n$, then a result of C. J. Smyth (Lemma~1 of~\cite{smyth}) implies that the conjugates $\alpha_1,\dots,\alpha_n$ are $\que$-linearly independent, and hence our lattice $L_f$ has rank $n$ in $n!$-dimensional Euclidean space~$K_{\real}$. Now, the fact that there exist infinitely many integer trinomials $f(x) = x^n + a_{n-1} x^{n-1} + a_0$ satisfying conditions~\eqref{tri} and $|a_{n-1}| > |a_0| + 1$ as in Theorem~\ref{thm1.5} with Galois group equal to $S_n$ follows, for instance, from~\cite{trinomial}.
\endproof

\bigskip

\bibliographystyle{plain}  

\begin{thebibliography}{10}

\bibitem{conway_sloane}
J.~H. Conway and N.~J.~A. Sloane.
\newblock{Sphere packings, lattices and groups, 3rd edition},
\newblock{Springer-Verlag, 1999.}

\bibitem{dub}
A. Dubickas.
\newblock On the degree of a linear form in conjugates of an algebraic number.
\newblock {\em Illinois J. Math.}, 46(2):571--585, 2002.

\bibitem{perm}
L. Fukshansky, S. R. Garcia and X. Sun.
\newblock Permutation invariant lattices.
\newblock {\em Discrete Math.}, 338(8):1536--1541, 2015.

\bibitem{lf_ek}
L. Fukshansky and E. Knight.
\newblock On lattices generated by algebraic conjugates of prime degree.
\newblock {\em Illinois J. Math.}, 70(1):193--211, 2026.

\bibitem{lf_dk}
L. Fukshansky and D. Kogan.
\newblock Cyclic and well-rounded lattices.
\newblock {\em Mosc. J. Comb. Number Theory}, 11(1):79--96, 2022.

\bibitem{macwillians_sloane}
F. J. MacWilliams and N.~J.~A. Sloane.
\newblock{The theory of error-correcting codes.},
\newblock{North-Holland Mathematical Library, Vol. 16. North-Holland Publishing Co., Amsterdam-New York-Oxford, 1977.}

\bibitem{mic1}
D.~Micciancio.
\newblock Generalized compact knapsacks, cyclic lattices, and efficient one-way functions from worst-case complexity assumptions.
\newblock {\em FOCS, IEEE Computer Society}, pages 356--365, 2002.

\bibitem{mic2}
D.~Micciancio.
\newblock Generalized compact knapsacks, cyclic lattices, and efficient one-way functions.
\newblock {\em Comput. Complexity}, 16(4):365--411, 2007.

\bibitem{trinomial}
H. Osada.
\newblock The Galois groups of the polynomials $X^n+aX^l+b$.
\newblock {\em J. Number Theory}, 25(2):230--238, 1987.

\bibitem{gqc}
I. Siap and N. Kulhan.
\newblock The structure of generalized quasi cyclic codes.
\newblock {\em Appl. Math. E-Notes}, 5:24--30, 2005.

\bibitem{smyth}
C. J. Smyth.
\newblock Additive and multiplicative relations connecting conjugate algebraic numbers.
\newblock {\em J. Number Theory}, 23(2):243--254, 1986.

\bibitem{d_wan}
D. Wan.
\newblock Euclidean SVP is NP-hard for Cyclic Lattices.
\newblock {\em arXiv:2609.16711v1}, 2026.

\end{thebibliography}

\end{document}